\documentclass[oneside,11pt]{article}

\usepackage{amsmath}
\usepackage{amssymb}
\usepackage{pxfonts}
\usepackage{dsfont}
\usepackage{graphicx}
\usepackage{eucal}
\usepackage{mathrsfs}
\usepackage{theorem}
\usepackage{pifont}
 \usepackage{sectsty}
\usepackage{amscd}
\usepackage{color}
\usepackage{fancyhdr}
\usepackage{framed}
\usepackage[all]{xy}
\usepackage[backref=page]{hyperref}

\definecolor{shadecolor}{rgb}{0.8,0.8,0.8}

\theoremheaderfont{\fontfamily{pzc}\bfseries\large}
\newtheorem{theorem}{Theorem}[section]

\newtheorem{lemma}[theorem]{Lemma}
\newtheorem{proposition}[theorem]{Proposition}

\newtheorem{definition}[theorem]{Definition}

\theorembodyfont{\rmfamily}

\newcommand{\specexercise}[1]{}

\newenvironment{proof}{{\flushleft \emph{Proof}:}}{\hfill\ding{110}}

\newcommand{\g}{\mathfrak{g}}

\newcommand{\M}{\mathcal{M}}
\newcommand{\R}{\mathbb{R}}
\newcommand{\N}{\mathbb{N}}
\newcommand{\Z}{\mathbb{Z}}

\newcommand{\vp}{\varphi}
\newcommand{\dist}{\operatorname{dist}}

\newcommand{\id}{\operatorname{Id}}
\newcommand{\supp}{\operatorname{supp}}
\newcommand{\Diff}{\operatorname{Diff}}
\newcommand{\Diffc}{\operatorname{Diff}_\text{c}}
\newcommand{\length}{\operatorname{length}}

\newcommand{\tL}[1]{\widetilde{L}_{#1}}

\newcommand{\pl}{\partial}

\newcommand{\beq}{\begin{equation}}
\newcommand{\eeq}{\end{equation}}

\newcommand{\brk}[1]{\left(#1\right)}          
\newcommand{\Brk}[1]{\left[#1\right]}          
\newcommand{\BRK}[1]{\left\{#1\right\}}        

\numberwithin{equation}{section}

\begin{document}

\title{Geodesic distance for right-invariant metrics on diffeomorphism groups: critical Sobolev exponents}
\author{Robert L.~Jerrard\footnote{Department of Mathematics, University of Toronto.} \,and Cy Maor\footnotemark[1]}
\date{}
\maketitle

\begin{abstract}
We study the geodesic distance induced by right-invariant metrics on the group $\Diffc(\M)$ of compactly supported diffeomorphisms of a manifold $\M$, and show that it vanishes for the critical Sobolev norms $W^{s,n/s}$, where $n$ is the dimension of $\M$ and $s\in(0,1)$.
This completes the proof that the geodesic distance induced by $W^{s,p}$ vanishes if $sp\le n$ and $s<1$, and is positive otherwise.
The proof is achieved by combining the techniques of two recent papers --- \cite{JM18} by the authors, which treated the subcritical case, and \cite{BHP18} of Bauer, Harms and Preston, which treated the critical 1-dimensional case.
\end{abstract}


\section{Introduction, preliminaries and main result}\label{sec_preliminaries}
The geometry of different diffeomorphism groups (e.g., compactly-supported, symplectic, volume-preserving) with respect to various right-invariant metrics has a long history (see, e.g., \cite{ER91,EP93,MM05,BHP18}).
One of the basic questions about these geometries is whether the geodesic distance induced by a given norm on the associated Lie algebra of the group actually generates a metric space structure on the group.
This may fail if two distinct diffeomorphisms can be connected with paths of arbitrary short lengths.

In this paper, we complete the full characterization of this vanishing geodesic distance phenomenon on the group of compactly-supported diffeomorphisms of a manifold, with respect to Sobolev norms $W^{s,p}$ on its Lie algebra of vector fields.
This study started in \cite{MM05}, and continued in \cite{BBHM13,BBM13}, where (among other results) the threshold $s=1/p$ between positive and vanishing geodesic distance was identified for one-dimensional manifolds.
In a recent paper \cite{BHP18} it was shown that the geodesic distance vanishes in this critical space, completing the characterization in the one-dimensional case.
Virtually simultaneously with \cite{BHP18}, in \cite{JM18} the authors identified the critical space in the $n$-dimensional case, namely $s=\min(n/p,1)$, leaving the case $sp= n$, $s<1$ open.
In this paper we combine the techniques of \cite{BHP18,JM18} to show that the geodesic distance vanishes in this case, thus completing the classification of vanishing geodesic distance phenomenon for compactly-supported diffeomorphisms.

\paragraph{Setting}
Let $(\M,\g)$ be a Riemannian manifold of \emph{bounded geometry}; that is, $(\M,\g)$ has a positive injectivity radius and all the covariant derivatives of the curvature are bounded: $\|\nabla^i R\|_\g < C_i$ for $i\ge 0$.
We denote by $\Diffc(\M)$ the group of compactly supported diffeomorphisms of $\M$, that is the diffeomorphisms $\vp$ for which the closure of $\{\vp(x)\ne x\}$ is compact, and by $\Gamma_c(T\M)$ the Lie-algebra of compactly supported vector fields on $\M$, the tangent space of $\Diffc(\M)$ at the identity. 

Given a norm $\|\cdot\|_A$ on $\Gamma_c(T\M)$, the length of a smooth path $\vp:[0,1]\to \Diffc(\M)$ is defined by
\[
\length_A\vp = \int_0^1 \|u_t\|_A\, dt, \qquad u_t  := \partial_t \vp_t \circ \vp_t^{-1}.
\]
Note that from the vector fields $\{u_t\}_{t\in [0,1]}$, and the initial condition $\vp_0$, the path $\vp$ can be recovered via standard ODE theory.

The above formula for lengths induces the \textbf{geodesic distance} between $\vp_0,\vp_1\in \Diffc(\M)$ in a standard way by
\[
\dist_A(\vp_0,\vp_1) := \inf\BRK{\length_A \vp \,\,:\,\, \vp:[0,1]\to \Diffc(\M), \, \vp(0) = \vp_0, \, \vp(1) = \vp_1}.
\]
Note that $\dist_A$ forms a semi-metric on $\Diffc(\M)$, that is, it satisfies the triangle inequality but may fail to be positive.
This paper is concerned exactly with this phenomenon --- for which Sobolev norms (defined below) the geodesic distance induces a metric space structure on $\Diffc(\M)$.

$\dist_A$ is, in fact, the geodesic distance of the \textbf{right-invariant Finsler metric} on $\Diffc(\M)$ induced by $\|\cdot\|_{A}$, which is defined as
\[
\|X\|_{\vp,A} := \|X\circ \vp^{-1}\|_A
\]
for every $\vp\in \Diffc(\M)$ and $X\in T_\vp \Diffc(\M)$.
If $\|\cdot\|_A$ is induced by an inner-product, it defines a Riemannian metric on $\Diffc(\M)$ in a similar manner;
many well-known PDEs are, in fact, the geodesic equations of such Riemannian metrics.
See \cite{BBHM13} for more details.
The right-invariance is inherited by $\dist_A$, as summarized in the following lemma:

\begin{lemma}[Right-invariance]
\label{lem:right_invariance}
For $\psi,\vp_0,\vp_1\in \Diffc(\M)$, we have 
\[
\dist_A(\vp_0 \circ \psi, \vp_1 \circ \psi) = \dist_A(\vp_0,\vp_1).
\]
In particular,
\[
\dist_A(\id,\psi) = \dist_A(\id,\psi^{-1}),
\]
and 
\[
\dist_A(\id,\vp_1 \circ \vp_0) \le \dist_A(\id, \vp_1) + \dist_A(\id, \vp_0).
\]
\end{lemma}

\begin{proof}
See \cite[Lemma~2.1]{JM18}.
\end{proof}

In this paper we are interested in fractional Sobolev $W^{s,p}$-norms, defined as follows:
\begin{definition}
\label{def:fractional_Sobolev}
For $0<s<1$ and $1\le p<\infty$, 
the $W^{s,p}$-norm of a function $f\in L^p(\R^n)$ is given by
\[
\|f\|_{s,p}^p = \| f\|_{L^p}^p + \int_{\R^n}\int_{\R^n} \frac {|f(x)-f(y)|^p}{|x-y|^{n+sp}}\, dx\,dy .
\]
\end{definition}

Given a Riemannian manifold $(\M,\g)$ of bounded geometry, this norm can be extended to $\Gamma_c(T\M)$ using a trivialization by normal coordinate patches on $\M$ (see \cite[Section~2.2]{BBM13} for details).
We will denote the induced geodesic distance on $\Diff_c(\M)$ by $\dist_{s,p}$. 
Different choices of charts result in equivalent metrics, and therefore the question of vanishing geodesic distance is independent of these choices.

Instead of using Definition~\ref{def:fractional_Sobolev} directly, we will bound the $W^{s,p}$-norm using the following interpolation inequalities:

\begin{proposition}[fractional Gagliardo--Nirenberg interpolation inequalities]
\label{pn:GN_inequality}
Assume that $1<p<\infty$. For every $f\in W^{1,p}(\R^n)$ and $s\in (0,1)$,
\[
\| f\|_{s,p} \le C_{s,p,n} \|f\|_{L^p}^{1-s} \|f\|_{1,p}^s\, , \quad\mbox{ where }\ \ 
\|f\|_{1,p}^p := \|f\|_{L^p}^p+ \|df\|_{L^p}^p,
\]
and
\[
\|f\|_{s,p} \le C_{s,p,n} \|f\|_{1,sp}^s \|f\|_{L^\infty}^{1-s}, \quad\mbox{ assuming }\ \ sp>1.
\]
\end{proposition}
For a proof, see \cite[Corollary~3.2]{BM01}.  
These are the only properties of the $W^{s,p}$-norm that will be used in this paper. 

\paragraph{Main results}
The main result of this paper is the following:
\begin{theorem}\label{main_thm}
Let $(\M,\g)$ be an $n$-dimensional Riemannian manifold of bounded geometry, and $p\in(n,\infty)$.
Then $\dist_{n/p,p}(\vp_0,\vp_1)= 0$ whenever $\vp_0,\vp_1$
belong to the  same path-connected component of $\Diffc(\M)$.
\end{theorem}

Combining this result with previous results, which are summed up in \cite[Theorem~2.4]{JM18}, we obtain the following full characterization of the vanishing geodesic distance phenomenon on compactly supported diffeomorphism groups:
\begin{theorem}
Let $(\M,\g)$ be an $n$-dimensional Riemannian manifold of bounded geometry.
Then for any $p\in[1,\infty)$, the induced $W^{s,p}$-geodesic distance vanishes on any path-connected component of $\Diffc(\M)$ if $sp\le n$ and $s<1$, and is strictly positive otherwise.
\end{theorem}

When $s>n/p$, then the Sobolev embedding $W^{s,p}\subset L^\infty$ implies that for every path $\{\vp_t\}_{t\in[0,1]}$ between $\vp_0,\vp_1\in\Diffc(\M)$, and every $x\in \M$,
\[
|\vp_1(x)-\vp_0(x)|  \le
\int_0^1 \left| \partial_t\vp_t(x)\right|\, dt
\le
\int_0^1 \| u_t\|_\infty\, dt
 \le
C\int_0^1 \| u(t)\|_{s,p} dt = C\length_{s,p}\vp,
\]
hence it is impossible to transport even a single point at a low cost.
On the other hand, when $sp \le n$, one expects to be able to transport small volumes over large distances at a small cost, using vector fields $u_t$ with $\|u_t\|_\infty \approx 1$ but $\|u_t\|_{s,p}\ll 1$.
Indeed, such vector fields are at the heart of all vanishing geodesic distance constructions on $\Diffc(\M)$ \cite{MM05,BBHM13,BHP18,JM18}.

The main difficulty in proving Theorem~\ref{main_thm}, compared with the subcritical case $s<\min\BRK{n/p,1}$ proved in \cite{JM18}, is that such vector fields are quite rigid in the critical case $sp=n$.
In the subcritical case, on the other hand, any function $f\in W^{s,p}(\R^n)$ can be rescaled $f_\lambda(x) := f( x/\lambda )$ with $\lambda \ll 1$ to obtain a function with the same $L^\infty$-norm but arbitrary small $W^{s,p}$-norm.
This rigidity in the critical case makes it difficult to control the endpoint of a path $\vp_t$ starting at $\vp_0$ and flowing along a vector field $u_t$ with these properties, and therefore it is difficult to construct arbitrary short paths between two {\em fixed} diffeomorphisms $\vp_0,\vp_1$.

In \cite{BHP18}, this problem is circumvented by using the notion of {\em displacement energy} defined in \cite{EP93}.
As described in the next section, they show that the geodesic distance vanishes if there exists an open set with zero displacement energy --- that is, if it is possible to transport the set so it does not intersect itself, for an arbitrary small cost.\footnote{Similar observations (in the context of contactomophorisms) also appear in \cite{She17}.}
This enabled them to prove Theorem~\ref{main_thm} in the one-dimensional case.
In this paper we combine this approach of using the displacement energy with the ideas used in \cite{JM18} to construct short paths in the subcritical case, to prove the vanishing of the geodesic distance in the critical case in every dimension.

The condition $s<1$ in Theorem~\ref{main_thm} is related to change, rather than transportation, of volumes.
That is, when $s\ge 1$ the $W^{s,p}$-norm detects any volume change, 
whereas when $s<1$ it is possible to have significant volume changes at a small cost, provided that no point moves very far.
When $n>1$, this plays an important role in constructing short paths, as will be clear from the proof.

Theorem~\ref{main_thm} is stronger than the main theorem of \cite{JM18}, as the latter proves vanishing geodesic distance only in the subcritical case.
Moreover, the proof of Theorem~\ref{main_thm} is significantly shorter, due to the fact that it is no longer needed to control of the endpoints of the short paths considered.
On the other hand, the proof of \cite{JM18}, being more direct, has the advantage of showing explicitly how two diffeomorphisms can be connected with arbitrary short paths, so in some sense it is more revealing or instructive.

\section{Displacement energy}

\begin{definition}
\label{def:displacement_energy}
The \textbf{displacement energy} of a set $V\subset \M$ with respect to the $W^{s,p}$-induced geodesic distance is defined by
\[
E(V) := \inf\BRK{\dist_{s,p}(\id, \vp) \,:\, \vp\in \Diffc(\M), \vp(V)\cap V = \emptyset }.
\]
\end{definition}

In this section we use \cite[Theorem~1]{BHP18} (see also \cite[Remark~7]{She17}, both generalize results of \cite{EP93}), to show that the $W^{s,p}$-geodesic distance vanishes if and only if there exists an open set $V\subset \M$ with $E(V) = 0$.

We start with the following lemma (which is almost identical to Step 2 in the proof of \cite[Theorem~2]{BHP18}):
\begin{lemma}
\label{lem:Lipschitz}
For every $s\in (0,1)$ and $p\in[1,\infty)$ and for every $\vp\in \Diffc(\M)$, the left multiplication operator $L_\vp: \Diffc(\M)\to \Diffc(\M)$, $L_\vp(\psi) = \vp\circ \psi$ is smooth and Lipschitz with respect to $\dist_{s,p}$.
\end{lemma}

\begin{proof}
The smoothness of $L_\vp$ is obvious.
We now prove that it is Lipschitz.
First, let $X\in \Gamma_c(T\M)$. Then
\[
\|dL_\vp X\|_{\vp,W^{s,p}} = \| dL_\vp X \circ \vp^{-1}\|_{s,p} = \|(d\vp(X))\circ \vp^{-1}\|_{s,p} \le C_\vp \|X\|_{s,p},
\]
for some $C_\vp>0$, by the continuity of multiplications and compositions, see Theorems 4.2.2 and 4.3.2 in \cite{Tri92}.
Now, let $\psi_0,\psi_1 \in \Diffc(\M)$, and let $\Psi:[0,1]\to \Diffc(\M)$ be a path between them.
Then $\vp\circ \Psi$ is a path between $\vp\circ \psi_0$ and $\vp\circ \psi_1$, and
\[
\begin{split}
\dist_{s,p}(\vp\circ \psi_0,\vp\circ \psi_1) 
	&\le \int_0^1 \|\pl_t(\vp\circ \Psi)\|_{\vp\circ \Psi,W^{s,p}} \,dt 
	= \int_0^1 \|dL_\vp \pl_t\Psi\|_{\vp\circ \Psi,W^{s,p}} \,dt \\
	& =  \int_0^1 \|dL_\vp (\pl_t\Psi\circ \Psi^{-1})\|_{\vp,W^{s,p}} 
	\le C_\vp \int_0^1 \|\pl_t\Psi\circ \Psi^{-1}\|_{s,p} \,dt.
\end{split}
\]
Taking the infimum on $\Psi$ we obtain
\[
\dist_{s,p}(\vp\circ \psi_0,\vp\circ \psi_1) \le C_\vp \dist_{s,p}(\psi_0,\psi_1),
\]
which completes the proof.
\end{proof}

Denote by $\Diff_0(\M)$ the connected component of the identity, i.e., all diffeomorphisms in $\Diffc(\M)$ for which there exists a curve between them and $\id$.
$\Diff_0(\M)$ is a simple group \cite{Eps70}.  
This fact, together with Lemma~\ref{lem:Lipschitz}, and the fact that $\Diffc(V)$ is non-Abelian for any open $V$, implies that the following corollary of \cite[Theorem~1]{BHP18} holds:
\begin{proposition}
\label{prop:displacement_energy}
There exists $\vp\in \Diff_0(\M)$, $\vp\ne \id$, such that $\dist_{s,p}(\id,\vp) = 0$ if any only if there exists an open set $V$ such that $E(V)=0$.
If such $\vp$ exists, then $\dist_{s,p}$ is identically zero on $\Diff_0(\M)$.
\end{proposition}

\section{Proof of Theorem~\ref{main_thm}}
The case $n=1$, $p=2$ was proved in \cite[Theorem~2]{BHP18}.
Their proof holds for every $p>1$, so here we prove for the case $n>1$.
It is enough to prove the result for $\R^n$ --- indeed, for a general manifold of bounded geometry $(\M,\g)$, one can embed the following $\R^n$ construction into a single coordinate chart, used in the definition of the induced $W^{s,p}$-geodesic distance on $\M$.

Since we will often split $\R^n = \R\times \R^{n-1}$, it is convenient to
write $m=n-1$.
We will denote the standard coordinates on $\R^n$ by $(x,y)$, where $x\in \R$ and $y\in \R^m$.

In the following lemma we construct functions $\xi_k\in W^{n/p,p}(\R^n)$, with $\|\xi_k\|_\infty = 1$ and $\|\xi_k\|_{n/p,p}\to 0$, for $p>n$.
That is, we bound the capacity of small balls in the critical Sobolev space $W^{n/p,p}(\R^n)$.
\begin{lemma}\label{lem:zero_cap}
Let $sp= n>1$, $s<1$, and let $(\lambda_k)_{k\in \N}$ be a sequence of positive numbers, $\lambda_k\ll e^{-k^p}$.
Then there exists a sequence $(\xi_k)_{k\in \N}$ of functions $\xi_k:\R^n\to [0,1]$ such that
\begin{enumerate}
\item $\xi_k\equiv 1$ on $[-\lambda_k,\lambda_k]^n$
\item $\supp \xi_k \subset [-1,1]^n$
\item $k^{n-1}\|\xi_k\|_{s,p} \to 0$.
\end{enumerate}
\end{lemma}

\begin{proof}
Let $r_k = \sqrt{n}\lambda_k$, so that $[-\lambda_k,\lambda_k]^n$ is contained in a ball of radius $r_k$.
Consider the function
\[
\xi_k(x) = 
\begin{cases}
1 							& |x|\le r_k 		\\
\frac{\log(1/|x|)}{\log(1/r_k)}	& |x|\in (r_k,1)	\\
0							& |x|\ge 1.

\end{cases}
\]
Then
\[
\| \xi_k\|_{L^n}^n \le |B_1(0)| = C(n)
\]
and $|d\xi_k|\le C \log(1/r_k)^{-1}/|x|$ for $|x|\in (r_k,1)$, and therefore
\[
\|d\xi_k \|_{L^n}^n \le C\log(1/r_k)^{1-n}.
\]
Hence
\[
\|\xi_k\|^n_{W^{1,n}} \le C\log(1/r_k)^{1-n}.
\]
Therefore, by Proposition~\ref{pn:GN_inequality}, we have
\[
\|\xi_k\|_{W^{n/p,p}} \le C \|\xi_k\|_{W^{1,n}}^{n/p} \|\xi_k\|_{L^\infty}^{1-n/p} \le C\log(1/r_k)^{(1-n)/p} \ll k^{(1-n)}.
\]
\end{proof}

Note that the above calculation is not optimal (one expects to be able to obtain $\|\xi_k\|_{{n/p},p}^p\approx\log(1/\lambda_k)^{1-p}$), but this simple construction is sufficient for our purposes.

\paragraph{General strategy of the proof:}
We now proceed to the proof of Theorem~\ref{main_thm}.
We prove it using Proposition~\ref{prop:displacement_energy}: we show that there exists an open set $U\subset \R^n$ whose displacement energy with respect to the $W^{n/p,p}$ norm is zero.
That is, we show that there exists a sequence $\Phi_k\in \Diffc(\R^n)$ such that $\Phi_k(U) \cap U= \emptyset$ and $\dist_{s,p} (\id,\Phi_k) \to 0$.
Specifically, we show this for the open set $U=(0,1)^n$.
In the rest of this section we construct these diffeomorphisms $\Phi_k$.

\paragraph{A sketch of the construction of the diffeomorphisms $\Phi_k$:}

Fix $k\in \N$.
	We consider $(0,1)^m$ as a union of sets $L_I$, $I=1,\ldots 2^m$, each $L_I$ is a union of $\approx  k^m$ disjoint cubes of diameter $\approx 1/k$.
	The main part of the proof consists of constructing diffeomorphisms $\Phi_k^I = (\phi_k^I(x,y) ,y)$, which satisfy 
	\[
	\lim_{k\to \infty}\dist_{s,p}(\id, \Phi_k^I) = 0, \qquad \phi_k^I(x,y)\ge x \qquad \text{and} \qquad \Phi_k^I ((0,1) \times L_I) \cap (0,1)^n = \emptyset.
	\]
	We then have that $\Phi_k = \Phi_k^{2^m} \circ \ldots \circ \Phi_k^1$ is the desired map.
	The construction of $\Phi_k^I$ is carried out in three stages:
	\[
	\Phi_k^I := \Psi_I^{-1} \circ \Theta_I \circ \Psi_I,
	\]
	where $\Psi_I$ and $\Theta_I$ (whose dependence of $k$ is omitted in order to simplify the notation) are as follows:
\begin{enumerate}
\item $\Psi_I(x,y) = (x,\psi_I(x,y)) $, squeezes each cube in $L_I$ to diameter $\lambda_k \ll e^{-k^p}$.
	Since $s<1$, this can be obtained at a small cost.
\item $\Theta_I(x,y) = (\theta_I(x,y) ,y)$ satisfies $\theta_I(0,y) = 1$ whenever $y$ is in one of the squeezed cubes. 
	Since $sp=n$, such a transport is possible at a low cost, but only if the volume of the transported points at every time is small enough; 
	this is the reason for the squeezing stage.
	$\theta_I$ is constructed (roughly) by flowing along translations of the vector field $u_t(x,y) = \xi_k(x-t,y)$, where $\xi_k$ are the maps constructed in Lemma~\ref{lem:zero_cap}.
\end{enumerate}

This scheme of splitting--squeezing--transporting--expanding is similar to the constructions in \cite{JM18}.
Since here we do not need to control the endpoint of the flow (just to transport $(0,1)^n$ away from itself), the transporting stage $\Theta_I$ is much simpler compared to \cite{JM18}.
On the other hand, the squeezing stage is somewhat more elaborate: In order for the norm of $\xi_k$ to be small, its support, which is a cube of diameter $\lambda_k$, needs to be small enough; in the subcritical case, it is enough to have $\lambda_k$ decay faster than any polynomial (in \cite{JM18} it is $\lambda_k \approx k^{-\log k}$), whereas here, in the critical case, we should have $\lambda_k \ll e^{-k^p}$, in view of Lemma~\ref{lem:zero_cap}.
Using the same squeezing strategy (i.e., same flow) as in \cite{JM18} for $\lambda_k \ll e^{-k^p}$ results in a path from $\id$ to a squeezing diffeomorphism $\Psi_I$ whose length is unbounded when $k\to \infty$ (as shown below), and so we need to alter this path in order to show that $\dist(\id,\Psi_I)$ tends to zero.

\paragraph{A detailed construction of the diffeomorphisms $\Phi_k$:}
We now construct $\Phi_k$ in full detail, and prove that $\dist_{s,p} (\id,\Phi_k) \to 0$.
Henceforth, all limits and asymptotic notations such as $o(1)$ are with respect to the limit $k\to \infty$.

\paragraph{Step I: splitting the cube into strips}

Fix $k\in \N$. We partition the lattice
$\frac{1}{k}\Z^m \subset \R^m$
into $2^m$  copies of $\frac{2}{k}\Z^m$:
\[
\frac 2k\Z^m,\, \frac2k\Z^m + \frac{e_1}k,\, \ldots,\, \frac2k\Z^m + \sum_{i=1}^m \frac {e_i}k,
\]
where $\{e_i\}_{i=1}^m$ is the standard basis of $\R^m$.
We index the different lattices as $Z_I$, $I\in \Z_2^m$, ordered by 
\[
(0,\ldots,0), (1,0,\ldots,0),  (0,1,0,\ldots,0), \ldots, (0,1,1,\ldots,1), (1,\ldots,1).
\]
Sometimes we will denote the indices by $1,\ldots, 2^m$ according to this order.
For each $I\in \Z^m_2$, denote 
\[
 L_I := \brk{Z_I + \Brk{-\frac{1}{2k},\frac{1}{2k}}^m}\cap [0,1]^m. 
\]
Note that $\cup L_I = [0,1]^m$. For $y\in \R^m$, we will write
\[
[y]_I := \mbox{ the closest point in $Z_I$ to $y$,}
\]
when a unique such point exists (such as when $y\in L_I$).

\paragraph{Step II: squeezing the strips}
Fix $1\le I \le 2^m$, and an auxiliary constant $\beta\in (0,1-s)$.
We now construct a diffeomorphism $\Psi_I\in \Diffc(\R^n)$, $\Psi_I(x,y) = (x,\psi_I(x,y))$, with
\beq
\label{eq:squeezing_cost}
\dist_{s,p}(\Psi_I,\id) = o(1),
\eeq
such that
\beq
\label{eq:squeezing}
\psi_I(x,y) = 2k\lambda_k(y-[y]_I) + [y]_I
\eeq
for every $x\in [0,1]$ and $y\in L_I$,
and with 
\beq
\label{eq:lambda_k_bound}
\lambda_k \ll \exp(-\exp(\beta k^\beta)) \ll \exp(-k^p).
\eeq
In particular, for every $x\in[0,1]$,
\beq
\label{eq:squeezed_strips}
\psi_I\brk{\BRK{x}\times L_I} = (Z_I + [-\lambda_k,\lambda_k]^m) \cap [0,1]^m =: \tL{I}.
\eeq

We construct the squeezing in two stages $\Psi_I = \Psi_I^2 \circ \Psi_I^1$. 
We show the construction for $I=1$; for $I\ne 1$ the construction is obtained by translating the $I=1$ case.

We start by constructing $\Psi_1^1$.
Let $u\in C_c^\infty((-1,1)^m;\R^m)$, such that $u(y) = -y$ for $y\in [-1/2,1/2]^m$, and extend it to a $2\Z^m$-periodic function on $\R^m$.
Let $\chi\in C_c^\infty(\R^n)$ such that $\chi\equiv 1$ on $[0,1]^n$.
Define $u_k^1(x,y) := \frac{\eta_k}{k} u(ky)\chi(x,y)$, where $\eta_k\gg 1$ will be fixed below. In particular, $u_k^1(x,y) = -\eta_k (y-[y]_1)$ for $x\in [0,1]$ and $y\in L_1$.

Note that 
\[
\|u_k^1\|_{L^p} \lesssim  \|u_k^1\|_{L^\infty} \lesssim \eta_k/k, \qquad \|d u_k^1\|_{L^p} \lesssim \|d u_k^1\|_{L^\infty} \lesssim \eta_k.
\]
Therefore, by Proposition~\ref{pn:GN_inequality} we have
\beq
\label{eq:squeezing_cost_1}
\|u_k^1\|_{s,p} \lesssim \frac{\eta_k^{1-s}}{k^{1-s}} \eta_k^s = \frac{\eta_k}{k^{1-s}} = o(1),
\eeq
where the last equality holds if we choose $\eta_k = k^\beta \ll k^{1-s}$ (recall that $\beta<1-s$).

Let $\psi^1(t,x,y)$ be the solution of 
\[
\pl_t \psi^1 = u_k^1(x,\psi^1), \qquad \psi^1(0,x,y) = y.
\]
Define $\psi^1_1(x,y) := \psi^1(1,x,y)$, and $\Psi^1_1(x,y) := (x,\psi^1_1(x,y))$.
A direct calculation shows that for  $(x,y)\in [0,1]\times [-1/2k, 1/2k]^m$, $\psi^1_1(x,y) = y e^{-\eta_k}$,
so by periodicity and the fact that $\chi\equiv 1$ on $[0,1]^n$, 
\beq
\label{eq:squeezing_first_stage}
\psi_1^1(x,y) = e^{-\eta_k}(y-[y]_1) + [y]_1
\eeq
for every $x\in [0,1]$ and $y\in L_1$.
Denote, for $x\in [0,1]$,
\[
\bar{L}_1 := \psi_1^1(\{x\}\times L_1).
\]
$\bar{L}_1$ is independent of $x$, and consists of $\approx k^m$ cubes of diameter $\approx \exp(-\eta_k)/k \ll \exp(-k^\beta)$.
Also, note that \eqref{eq:squeezing_cost_1} implies that 
\[
\dist_{s,p}(\Psi^1_1,\id) = o(1).
\]

Note that we cannot choose $\eta_k$ to be large enough such that $\bar{L}_1$ consists of cubes of diameter $\ll \exp(-k^p)$, which is our ultimate goal here; indeed, this would require $\eta_k \approx k^p \gg k^{1-s}$, which violates \eqref{eq:squeezing_cost_1}.
However, once we squeeze $L_1$ into $\bar{L}_1$, we can start a new squeezing stage that only squeezes $\bar{L}_1$.
That is, instead of having a vector field $u$ that satisfies $u(x,y) = -\alpha(y-[y]_1)$ for $y\in L_1$ (where $\alpha>0$ is a constant), we only need this to hold for $y\in \bar{L}_1$.
Since $\bar{L}_1$ is much smaller than $L_I$, we can have a much larger squeeze factor $\alpha$, while keeping the norm of $u$ small.
This second squeezing stage that is described below.

We denote the second squeezing stage $\Psi^2_1 = (x,\psi_1^2(x,y))$.
Again, we define $\psi_1^2(x,y) = \psi^2(1,x,y)$, where $\psi^2(t,x,y)$ is the solution of 
\[
\pl_t \psi^2 = u_k^2(x,\psi^2), \qquad \psi^2(0,x,y) = y,
\]
for $u_k^2(x,y)$ that satisfies $u_k^2(x,y) = -\alpha_k(y-[y]_1)$ for $x\in [0,1]$, $y\in \bar{L}_1$, and $\alpha_k\gg 1$ that will be fixed below.
Since $\bar{L}_1$ consists of cubes of diameter $\ll \exp(-k^\beta)$, we can choose $u_k^2$ such that 
\[
\|u_k^2\|_p  \lesssim \|u_k^2\|_\infty \ll \alpha_k\exp(-k^\beta), \qquad 
\|u_k^2\|_p \lesssim\|du_k^2\|_\infty \approx \alpha_k.
\]
Choosing $\alpha_k = \exp(\beta k ^\beta)$, we obtain, since $\beta<1-s$, that
\[
\|u_k^2\|_{s,p} \ll \alpha_k \exp(-(1-s)k^\beta) = o(1).
\]
In particular we have that 
\[
\dist_{s,p}(\Psi^2_1,\id) = o(1).
\]
It follows that $\psi_1^2(x,\cdot)$ squeezes $\bar{L}_1$ by a factor of $\exp(-\alpha_k) = \exp(-\exp(\beta k^\beta))$, that is 
\beq
\label{eq:squeezing_second_stage}
\psi_1^2(x,y) = \exp(-\exp(\beta k^\beta))(y-[y]_1) + [y]_1
\eeq
for every $x\in [0,1]$ and $y\in \bar{L}_1$.
Therefore
$\Psi_1 = \Psi_1^2 \circ \Psi_1^1$ squeezes $L_1$ such that \eqref{eq:squeezing}-\eqref{eq:lambda_k_bound} hold, with $\lambda_k = \exp(-\exp(\beta k^\beta)- k^\beta)/2k$.
By Lemma~\ref{lem:right_invariance}, we have 
\[
\dist_{s,p}(\Psi_1,\id) \le \dist_{s,p}(\Psi^2_1,\id) +\dist_{s,p}(\Psi^1_1,\id) = o(1),
\]
as required.

\paragraph{Step III: Flowing the squeezed strips}
Recall that by \eqref{eq:squeezed_strips}, $\lambda_k \ll e^{-k^p}$ is the width of the squeezed strips $\tL{I}$ defined by \eqref{eq:squeezed_strips},
and let $\xi_k$ be the function associated with $\lambda_k$ as defined in Lemma~\ref{lem:zero_cap}. 
Define 
\[
\xi_k^I(x,y) := \sum_{z\in Z_I\cap [0,1]^m} \xi_k(x,y-z )
\]
and
\[
v_k(t,x,y) =  \xi_k^I(x-t,y).
\]
Note that\footnote{The righthand side inequality in \eqref{eq:cost_transport_vector_field} is the reason we need $\lambda_k$ to be so small, which is achieved by the two-stage squeezing.
In the subcritical case $sp<n$, the $W^{s,p}$-capacity of small balls is much smaller, hence $\lambda_k$ can be larger (that is, the results of Lemma~\ref{lem:zero_cap} hold for larger values of $\lambda_k$), and then the one-stage squeezing used in \cite{JM18} suffices.}
\beq
\label{eq:cost_transport_vector_field}
\|v_k(t,\cdot)\|_{s,p} = \|\xi_k^I\|_{s,p} \le \sum_{z\in Z_I\cap [0,1]^m} \|\xi_k\|_{s,p} \lesssim k^m \|\xi_k\|_{s,p} = o(1).
\eeq

Let $\theta_I(t,x,y)$ be the solution of 
\[
\pl_t\theta_I = v_k(t,\theta_I,y), \qquad \theta_I(0,x,y) = x,
\]
and define $\Theta_I(t,x,y) = (\theta_I(t,x,y),y)$.
Denote $\theta_I(x,y) := \theta_I(1,x,y)$ and $\Theta_I(x,y) := \Theta_I(1,x,y)$.

Note that for $y\in \tL{I}$, we have $\xi_k^I(0,y) \ge 1$, and therefore $\theta_I(t,0,y) \ge t$.
Since $\theta_I(t,x',y) > \theta_I(t,x,y)$ whenever $x'>x$, we have that 
\beq
\label{eq:transport}
\theta_I(x,y) > \theta_I(0,y) \ge 1, \qquad \text{for every $x>0$ and $y\in \tL{I}$.}
\eeq
Note also that since $\xi_k\ge 0$, we have that 
\beq
\label{eq:transport_2}
\theta_I(x,y) \ge x, \qquad \text{for every $(x,y)$.}
\eeq
Finally, \eqref{eq:cost_transport_vector_field} implies that
\beq
\label{eq:cost_transport}
\dist_{s,p}(\Theta_I,\id) = o(1).
\eeq

\paragraph{Step IV: conclusion of the proof}

Now, define 
\[
\Phi_k := \Phi_k^{2^m} \circ \ldots \circ \Phi_k^1, \qquad
\Phi_k^I := \Psi_I^{-1} \circ \Theta_I \circ \Psi_I.
\]
Note that $\Phi_k$ and $\Phi_k^I$ only change the $x$ coordinates; therefore, we write
\[
\Phi_k(x,y) = (\phi_k(x,y),y), \qquad
\Phi_k^I(x,y) = (\phi_k^I(x,y),y).
\]
Estimates \eqref{eq:squeezing_cost} and \eqref{eq:cost_transport}, together with Lemma~\ref{lem:right_invariance} imply that
\[
\dist_{s,p}(\id, \Phi_k) = o(1).
\]
We now claim that $\Phi_k(U) \cap U = \emptyset$. This will complete the proof as it shows that the displacement energy of $U$ is zero, since $\Phi_k(U) \cap U = \emptyset$ implies that $E(U) \le \dist_{s,p}(\id,\Phi_k)$ and the righthand side tends to zero.

Let $(x,y)\in U$.
In particular, $y\in L_I$ for some $I$. 
Therefore, $\psi_I(x,y) \in \tL{I}$, and therefore, since $x>0$, we have from \eqref{eq:transport}-\eqref{eq:transport_2} that
\[
\phi_k(x,y) \ge \phi_k^I(x,y) = \theta_I(x,\psi_I(x,y)) > 1,
\]
hence $\Phi_k(x,y) \notin U$, hence $\Phi_k(U) \cap U = \emptyset$.

\paragraph{Acknowledgements}
We are grateful to Martin Bauer, Philipp Harms and Stephen Preston for introducing us their paper and the notion of displacement energy.
This work was partially supported by the Natural Sciences and Engineering Research Council of Canada under operating grant 261955.


{\footnotesize
\bibliographystyle{amsalpha}
\providecommand{\bysame}{\leavevmode\hbox to3em{\hrulefill}\thinspace}
\providecommand{\MR}{\relax\ifhmode\unskip\space\fi MR }
\providecommand{\MRhref}[2]{%
  \href{http://www.ams.org/mathscinet-getitem?mr=#1}{#2}
}
\providecommand{\href}[2]{#2}

}
\end{document}